\documentclass[11pt]{amsart}

\usepackage{amssymb,amsrefs,amscd,skak,xcolor}
\usepackage{tikz}
\usepackage{tikz-cd}
\usepackage{bm}
\usetikzlibrary{matrix,arrows}

\usepackage{hyperref}

\usepackage[a4paper,twoside,left=2.8cm,right=2.8cm,top=3.1cm,bottom=2.3cm]{geometry}
\usepackage{lmodern}
\usepackage{multirow}

\BibSpec{article}{%
  +{}{\PrintAuthors}  		{author}
  +{,}{ \textit}     		{title}
  +{,}{ }             		{journal}
  +{}{ \textbf}       		{volume}
  +{}{ \parenthesize} 		{date}
  +{,}{ }      	      		{conference}
  +{,}{ }      	      		{book}
  +{,}{ }            		{pages}
  +{,}{ }            	 	{note}
  +{,}{ }            	 	{status}
  +{, available at}{  \texttt } {eprint}
  +{.}{}              {transition}
}
\BibSpec{book}{%
  +{}{\PrintAuthors}  {author}
  +{,}{ \textit}      {title}
  +{}  { \PrintEditorsB} 		    {editor}
  +{,} { }            {series}
  +{,} { \voltext}    {volume}
  +{,} { }            {publisher}
  +{,}{ }             {date}
  +{,}{ }	      {note}
  +{.}{}              {transition}
}

\BibSpec{collection.article}{%
+{}  {\PrintAuthors}                {author}
+{,} { \textit}                     {title}
+{.} { }                            {part}
+{:} { \textit}                     {subtitle}
+{,} { \PrintContributions}         {contribution}
+{,} { \PrintConference}            {conference}
+{}  {\PrintBook}                   {book}
+{,} { }                            {booktitle}
+{}  { \PrintEditorsB} 		    {editor}
+{,} { }            		    {series}
+{,} { \voltext}    		    {volume}
+{,} { }          		    {publisher}
+{,} { \PrintDateB}                 {date}
+{,} { pp.~}                        {pages}
+{,} { }                            {status}
+{.} {}                             {transition}
}

\newtheorem{theorem}{Theorem}[section]
\newtheorem*{theorem*}{Theorem}

\theoremstyle{plain}

\newtheorem{lemma}[theorem]{Lemma}
\newtheorem{proposition}[theorem]{Proposition}

\newtheorem*{lemma*}{Lemma}
\newtheorem*{proposition*}{Proposition}

\theoremstyle{definition}
\newtheorem{definition}[theorem]{Definition}

\newcommand{\NN}{\mathbb{N}}

\newcommand{\RR}{\mathbb{R}}

\newcommand{\EE}{\mathbb{E}}

\newcommand{\calL}{\mathcal{L}}

\newcommand{\calV}{\mathcal{V}}
\newcommand{\calLK}{\mathcal L\mathcal K}
\newcommand \restrict[2]{\left. #1\right|_{#2}}

\newcommand{\inv}{^{-1}}

\newcommand \sS{\mathbb S}
\newcommand \U{\mathbb U}

\DeclareMathOperator{\tube}{tube}

\DeclareMathOperator{\dist}{dist}

\DeclareMathOperator{\vol}{vol}

\DeclareMathOperator{\nc}{nc}
\DeclareMathOperator{\sym}{sym}
\DeclareMathOperator{\abssym}{abssym}
\DeclareMathOperator{\Id}{Id}
\DeclareMathOperator{\PR}{PR}
\DeclareMathOperator{\Gr}{Gr}
\DeclareMathOperator{\Pluck}{{Pl\ddot u}}
\title{On the Gaussian Kinematic Formula of R. Adler and J. Taylor}
\date{\today}
\author {Joseph H.G. Fu}
\email{joefu@uga.edu}
\address{Department of Mathematics, University of Georgia}

\begin{document}

\maketitle
\begin{abstract} We apply methods of algebraic integral geometry to prove a special case of the Gaussian kinematic
formula of Adler-Taylor. The idea, suggested already by Adler and Taylor, is to view the GKF as the limit of spherical
kinematic formulas for spheres of large dimension $N$ and curvature $\frac 1 N$.
\end{abstract}

\section{Introduction}
In this paper we formulate and prove a theorem of R. Adler and J. Taylor about excursion
sets of random functions on the sphere, in terms of the theory of smooth valuations in the sense of
S. Alesker. The valuations that arise are of two types. For simplicity we consider valuations 
to be defined on sets with positive reach (PR sets) throughout.
\subsection{Two families of valuations}
The first type consists of  canonical valuations associated to any smooth Riemannian manifold $M$. Referring to \cite {crm}, 
there exists an algebra $\calLK(M)$ of smooth valuations, generated 
by a single canonical element $t$ such that $t^{\dim M}$ is (up to scale) the volume measure and $t^{\dim M +1}=0$. The various powers $t^k$ 
are the total integrals of the  {\it Lipschitz-Killing curvatures}, in the sense that if $A \subset M$
is a compact submanifold with boundary then
$$
t^k(A) = \int_A \phi_k + \int_{\partial A} \psi_k
$$
for suitable differential forms $\phi_k$, $\psi_k$ that are polynomials in the curvature and connection forms (cf. e.g. (7), (8) of \cite{fu_wan}).
In particular, the multiplicative unit element $t^0= \chi$, the Euler characteristic. These valuations are homogeneous: $t^k(\{rx: x\in A\})= r^kt^k(A)$.

The second type consists of the {\it Gaussian intrinsic volumes} $\gamma_0,\gamma_1,\dots$ on a 
euclidean space $\RR^d$. For measurable $D \subset \RR^d$ put $\gamma_0(D)$ to be the volume of $D$ with respect to the standard unit covariance
Gaussian measure on $\RR^d$:
$$
\gamma_0(D) := \frac 1{\sqrt {2\pi d}} \int_D e^{-\frac {|x|^2} 2}\, dx.
$$
To define $\gamma_i(D)$ for $i>0$ we restrict to  $D\in \PR$. Let $\tube(D,r)$ be the set of points within distance $r$ of $D$. Then
\begin{equation}
\gamma_i(D):= \restrict{\frac{d^i}{dr^i}}{r=0} \gamma_0(\tube(D,r)).
\end{equation}

\subsection{Statement of the theorem}
The {\it canonical process} on the unit sphere $ \{(x_0,\dots, x_n): \sum x_i^2 = 1\}= :\sS_1^n \subset \RR^{n+1}$ is the random function 
\begin{equation}
f(x) = f_\omega(x) = \sum_{i=0}^n x_i \xi_i(\omega)
\end{equation}
where $\xi_0,\dots,\xi_n$ are independent Gaussian random variables with mean zero and 
variance 1.  Let $f_1, \dots, f_d$ be independent copies of the canonical process on $\sS_1^n$, and put
$F:= (f_1,\dots,f_d): \sS_1^n \to \RR^d$.

\begin{theorem}\label{thm_main}
\begin{equation}\label{eq_adler_taylor}
\EE[t^m(A \cap F \inv D])]= \sum_{k} \frac {(\frac \pi 2)^{\frac k 2} }{ k! \omega_k }t^{k+m}(A) \gamma_k(D) 
\end{equation}
for any compact sets $A \in \PR( \sS_1^n), D \in\PR( \RR^d)$.
\end{theorem}

In fact this theorem is a special case of a much broader theorem \cite{adl_tay_book} of Adler-Taylor. Our proof of the special case follows their outline \cite{adl-tay}, which is based on a beautiful 
connection with spherical integral geometry via ``Poincar\'e's limit" (cf. \eqref{eq_poincare} below).
Our contribution  is to apply methods of algebraic 
integral geometry to execute their proof scheme.

%
%

\subsection{Acknowledgements} I wish to thank Robert Adler and Gil Solanes for helpful 
discussions. I am also extremely grateful to Andreas Bernig, Julius Ross and Thomas Wannerer 
for the opportunity to speak on this subject at the
 workshop they organized at the  MFO in November 2023, which provided a crucial  spur to complete this paper.
 
 I am grateful as well to support from the Simons Foundation.
 
\section{Notation and basic facts}
The unit ball of radius $r$ in $\RR^k$ is denoted $B^k_r$.
We denote the sphere $\partial B^{k+1}_r$  of dimension $k$ and radius $r$ by $\sS^k_r$ and
put
$$
\Sigma^N:= \sS^N_{\sqrt N}.
$$
We will use the notation
$$
\Sigma^N_r
$$
to denote a geodesic ball of radius $r$ in $\Sigma^N$. 

Throughout the paper we will fix the dimension $d$ of the euclidean space that includes $D$. For positive functions $a, b$, we write
$$
a\precsim b
$$
if there is a constant $ C>0 $ (depending only on $d$ and $D$) such that 
$$a < Cb ,
$$
and 
$$
a\asymp b \iff a\precsim b \precsim a.
$$
If $a= a(n), b = b(n) , \ n\in \NN$, we write
$$
a\sim b 
$$
if
$$
\lim_{n\to \infty } \frac {a(n)}{b(n)} = 1.
$$
Thus $a\sim b \implies a\asymp b$.

We will rely heavily on Stirling's formula
\begin{equation}
\Gamma(k+1) \sim \sqrt{2\pi k }   \left(\frac k e\right)^k
\end{equation}
and the classical limit
$$
 \lim_{k\to \infty}\left(1+ \frac x k\right)^k = e^x
$$
where the sequence increases monotonically if $x >0$.
We obtain an asymptotic expression  for the volume $\omega_n$ of the unit ball in $\RR^n$
\begin{equation}\label{eq_omega}
\omega_n = \frac {\pi^{\frac n 2}}{\Gamma(\frac n 2 + 1)}
\sim {\frac 1 {\sqrt{\pi n}}} \left(\frac{2\pi e}n\right)^{\frac n 2}  \implies \frac{\omega_n}{\omega_{n-1}}\sim \sqrt{
\frac{2\pi}n}
\end{equation}
and of the $n$-dimensional unit sphere
\begin{equation}
\alpha_n:= (n+1) \omega_{n+1} \sim \sqrt 2 \left(\frac{2\pi e} n\right)^{\frac n 2}
\end{equation}
as well as  the binomial coefficients
\begin{equation}\label{eq_binom} 
\binom n k
 \asymp  \frac{{n^{n+\frac 1 2}} }{{k^{k+\frac 1 2}} {(n-k)^{n-k+\frac 1 2}} }\precsim
\left(\frac {ne} k\right)^k\sqrt{\frac{n}{k(n-k)}}
\end{equation}

Recall also from \cite{kl-ro} the values of the intrinsic volumes (cf. Section \ref{sect_canon} below) of the 
unit euclidean ball of dimension $N$:
\begin{align}\label{eq_intrinsic_vol}
\mu_k(B^N_1) & = \frac {\omega_N }{\omega_{N-k}}\binom N k\\
\notag& 
\sim \frac 1 {\sqrt{2\pi k}}\left(\frac e k\right)^k (2\pi N)^{\frac k 2}
\end{align}
for fixed $k$ as $N \to \infty$, and $\precsim $  in general.

\section{The scheme}
\subsection{Another Poincar\'e limit}\label{sec_approx} It will be helpful to view the matter in terms of probability densities on the space $\calL$  of linear maps $\RR^{n+1}\to \RR^d$, since the maps $F$
are restrictions of such maps to $\sS^n_1$. Identifying $\calL$ with the euclidean space of $(n+1) \times d$ real matrices, put $dF$ for the standard Lebesgue measure on $\calL$ and set $\Pi_\infty$ for the standard Gaussian density on $\calL$ centered at $0$ and with  covariance $I_d \otimes I_{n+1}$. The random map $F \in \calL$  of  Theorem \ref{thm_main} is chosen
according to the distribution $\Pi_\infty$, and may be written
$$
\RR^{n+1} \owns y\mapsto \left(\sum_iy_i\xi_{i1} ,\dots,\sum_iy_i\xi_{id}\right)
$$
where the $\xi_{ij}$ are independent centered Gaussian random variables with unit variance.

The analytic core of the proof scheme of \cite{adl-tay} rests on the following approximation to $\Pi_\infty$. 
For each $N \ge n$, consider the  density $\Pi_N$ on $\calL$ given as the image of the Haar probability measure on $O(N+1)$ under the map that takes $g \in O(N+1) $ to $\sqrt N$ times the upper left $d\times (n+1)$ block of $g$:
$$
g\mapsto \sqrt N\left(\begin{matrix}g_{10}& g_{11}& \dots & g_{1n}\\
g_{20}& g_{21}& \dots & g_{2n}\\
\dots & \dots &\dots & \dots \\
g_{d0}& g_{d1}& \dots & g_{dn}\\
\end{matrix}\right)
$$
Then  $\Pi_N\to \Pi_\infty$ in a sense that we will make precise in Section \ref{sect_Pi_converge} below. Denote by $\EE_N$ the expectation with respect to the distribution $\Pi_N$, and by $\EE_\infty$ the expectation with respect to $\Pi_\infty$. The first  step in proving Theorem \ref{thm_main} is to show
for compact PR subsets $A \subset \sS^n_1$ and $D\subset \RR^d$
\begin{equation}\label{eq_fund_limit}
\lim_{N\to \infty} \EE_N[t^m(A \cap F\inv(D))]= \EE_\infty [t^m(A \cap F\inv(D))].
\end{equation}

\subsection{Enter the spherical kinematic formula} \label{sect_sphere} 
Consider the embedding $\iota_N:\sS^{n}_1 \hookrightarrow \Sigma^N$ given by
$$\iota_N (y_0,\dots, y_n) := (\sqrt N y_0,\dots, \sqrt N y_n, 0\dots, 0 ) \in \RR^{N+1}.$$
Put $\Sigma^N \subset \RR^{N+1}$ for the sphere of curvature $\frac 1 N$ (i.e. radius $\sqrt N$) and set
\begin{equation}
D_N:= \Sigma^N \cap \pi_N\inv D
\end{equation}
where $\pi_N$ is the orthogonal projection $\RR^{N+1}\to \RR^d$.
Then the restriction to $\sS^n_1$ of the  random map $F$, selected according to the distribution $\Pi_N$, is
$$
F= \pi_N \circ g \circ \iota_N
$$
so  if $dg$ is the probability Haar measure on $O(N+1)$ then for any PR set $A \subset \sS_1^n$
\begin{align}
\notag\EE_N [t^m(A \cap F\inv D)]] &= \int_{O(N+1)} t^m(A\cap \iota_N\inv g\inv\pi_N\inv D) \, dg\\
\notag&= \int_{O(N+1)} t^m(\iota_N\inv(\iota_NA\cap  g\inv\pi_N\inv D)) \, dg\\
\label{eq_KF1} &=N^{-\frac m 2} \int_{O(N+1)} t^m(\iota_N A\cap  g\inv D_N) \, dg\\
\notag&= 2^m\int_{O(N+1)} \left(\frac t{\sqrt {4N}}\right)^m(\iota_N A\cap  g\inv D_N) \, dg.
\end{align}

\begin{definition}
It will be convenient to set
\begin{equation}
u:= \frac t {\sqrt{4N}}.
\end{equation}
\end{definition}

Denote by $\calV^N$  the $(N+1)$-dimensional space of $O(N+1)$-invariant valuations on $\Sigma^N$.  
The kinematic operators $k_N: \calV^N \to \calV^N\otimes \calV^N$ from \cite{crm} are given with respect to the normalization of the  Haar measure $dg$ on  $O(N+1)$ determined by 
$$
dg(\{g: g o \in E\})= \vol (E) , \quad E \subset \Sigma^N.
$$ 
Thus
$$
p_N:= N^{-\frac N 2}  \alpha_{N}\inv k_N
$$
is the kinematic operator corresponding to
 the probability Haar measure. That is, if $\mu \in \calV^N$ and $dg_{prob}$ is the probability Haar measure
 then we write 
 $$
 p_N(\mu)= \sum_i \alpha_i \otimes \beta_i
 $$
 to mean 
 $$
 \int_{O(N+1)} \mu(A \cap gB) \, dg_{prob} = \sum_i \alpha_i(A) \beta_i(B) 
 $$
 where the $\alpha_i,\beta_i \in \calV^N$.
 
 Thus we may rewrite \eqref{eq_KF1} as
 \begin{equation}\label{eq_KF2}
\EE_N [t^m(A \cap F\inv D) ] = 2^mp_N(u^m)(\iota_N A,  D_N)  .
\end{equation}

Since $p_N(\chi)$ is a nonsingular element of $\calV^N\otimes \calV^N$ (\cite{crm}), there is a basis $\nu_0,\nu_1,\dots,\nu_N $ for $\calV^N$ such that 
\begin{equation}\label{eq_pN_chi}
p_N(\chi) = \sum_k u^k\otimes \nu_k
\end{equation}
and by multiplicativity  
$$
p_N(u^m)= \sum_k u^{m+k}\otimes \nu_k
$$
For notational simplicity  we have muted  the fact that the $\nu_k$ depend on $N$.
Thus \eqref{eq_KF1} becomes
\begin{align*}
\EE_N[t^m(A\cap F\inv D)]&= 2^m p_N(u^m)(\iota_NA, D_N) \\
&= 2^m\sum_k  u^{m+k}(\iota_N A) \nu_k(D_N)\\
&= \sum_k  2^{-k}t^{m+k}(A) \nu_k(D_N)
\end{align*}
The sum is finite since $A \subset \sS^n_1$, and therefore $t^i(A) =0$ for $i>n$.

The second and concluding step in proving Theorem \ref{thm_main} will be to show
 \begin{equation}\label{eq_limit_of_nu}
\lim_{N\to\infty} \nu_k(D_N) =  \frac {( 2 \pi )^{\frac k2}}{k! \omega_k }\gamma_k(D)
\end{equation}
 for each fixed $k$.
 
 \section{Valuations and kinematic formulas on $\Sigma^N$} 
 The proofs of both \eqref{eq_fund_limit} and \eqref{eq_limit_of_nu} rely on a detailed understanding of spherical kinematic formulas. The account below is based on \cite{crm}.

 \subsection{Canonical bases for $\calV^N$}\label{sect_canon} Recall from \cite{crm} the canonical bases 
 \begin{equation}
 \phi^i , \ t^i, \ \mu_i, \ \tau_i , \quad 0\le i \le N 
\end{equation}
for $\calV^N$. The first two are Alesker powers of 
\begin{equation}
\phi := \sqrt{4N} \int_{O(N+1)} \chi( \cdot \cap g \sS_{\sqrt N}^{N-1})\, dg
\end{equation}
where $\sS_{\sqrt N}^{N-1}$ is a great hypersphere of $\Sigma^N$, and of the canonical generator
\begin{equation}\label{eq_phi_to_t}
t= \frac \phi{\sqrt{1- \frac{\phi^2}{4N}}}
\end{equation}
of the Lipschitz-Killing algebra of $\Sigma^N$, respectively. The $\mu_i$ are the classical 
intrinsic volumes, viz. the scalar multiples of the respective $t^i$ indicated in the relation
\begin{equation}\label{eq_t_mu}
e^{\pi t} = \sum_i \omega_i \mu_i.
\end{equation}
They satisfy the relations
\begin{equation}\label{eq_mu1}
\dim A =k \implies \mu_k(A) = \vol_k(A)
\end{equation}
and
\begin{equation}\label{eq_product}
\mu_k(A\times B) = \sum_{i+j=k}\mu_i(A)\mu_j(B)
\end{equation}
for a metric Cartesian product $A \times B$.

 The fourth basis is most readily described 
by the identity 
\begin{equation}
\frac 1 2 \tau_i(\sS_{\sqrt N}^{j}) =  (4N)^{\frac i 2} \delta^i_j
\end{equation}
(cf. \cite{crm}, p. 92) where $\sS_{\sqrt N}^{j}\subset \Sigma^N$ is a great $j$-sphere. Explicitly,
\begin{equation}\label{eq_def_tau}
\tau_i (A) = 2^{i+1} \alpha_i\inv \alpha_{N-i-1}\inv \int_{\partial A} \sym_{N-i-1} \, d\vol_{N-1}
\end{equation}
for smooth domains $A \subset \Sigma^N$, where $\sym_j$ denotes the $j$th elementary symmetric 
function of the principal curvatures of $\partial A$.
It is convenient also to define
\begin{equation}\label{eq_def_sigma}
\sigma_i:= \frac {\tau_{N-i}}{(4N)^{\frac {N-i} 2}}=  2{N^{-\frac {N-i} 2}}\alpha_{N-i}\inv \alpha_{i-1}\inv \int_{\partial A} \sym_{i-1} \, d\vol_{N-1}.
\end{equation}

These bases are related by  \eqref{eq_phi_to_t} and
\begin{align}
\label{eq_t_to_phi}\phi&= \frac t{\sqrt{1+ \frac{t^2}{4N}}}= \sqrt{4N} \frac u{\sqrt{1+ u^2}}\\
\tau_i&= \phi^i\left(1- \frac {\phi^2}{4N}\right) 
\iff \sigma_{N-i} =\frac{u^i}{\left(1+u^2\right)^{\frac i 2 + 1}} 
\end{align}
In order to invert the last relation we compute (cf. \cite{wilf})
\begin{align}
\notag u^k&= \frac{u^k}{(1+u^2)^{\frac k 2 +1}}\left(\frac 1{1-\frac{u^2}{1+u^2}}\right)^{\frac k 2 +1}\\
\label{eq_sigma_to_u}&= \frac{u^k}{(1+u^2)^{\frac k 2 +1}}\left(\sum_j \binom{j+\frac k 2} j \left(\frac{u^2}{1+u^2}\right)^j\right)\\
\notag &=\sum_j \binom{j+\frac k 2} j \sigma_{N- (k+2j)}
\end{align}


\subsection{Some kinematic formulas} From \cite{crm} we find that the probability kinematic operator $p_N$ is given by 
\begin{equation}\label{eq_tau_KF}
p_N(\tau_k) = 2^{-N-1}N^{-\frac N 2} \sum_{i+j = N+ k} \tau_i \otimes \tau_j
\end{equation}
or
\begin{equation}
p_N(\sigma_k) = \frac 12 \sum_{i+j =  k} \sigma_i \otimes \sigma_j.
\end{equation}
In the $\phi$ basis,
\begin{equation}
p_N(\phi^k) = 2^{-N-1}N^{-\frac N 2}  \sum_{i+j= N+k} \phi^j \otimes \tau_i.
\end{equation}
In particular, recalling \eqref{eq_t_to_phi},
\begin{equation*}
p_N(\chi) = \frac   12\sum_{i} \left({\frac u{\sqrt{1+u^2}}}\right)^{i} \otimes \sigma_i
\end{equation*}

From this formula we may express the $\nu_k$ from \eqref{eq_pN_chi} in terms of the $\sigma_j$, as follows. Introduce the 
formal symbol $x$ so that  $x^i $ corresponds to $\sigma_i$, and regard $p_N(\chi)$ as an element of the
 ring of polynomials in $x$ with coefficients in $\calV^N$ (recall that $u^k=0=\sigma_k$ for $k >N$). Then
\begin{align*}
2p_N(\chi) &= \sum_{i} \left({\frac {ux}{\sqrt{1+u^2}}}\right)^{i}  \\
&=  \left(1- \frac {ux}{\sqrt{1+u^2}} \right)\inv\\
& = [{(1+u^2)  + (u x\sqrt{1+u^2} )}][{ 1 -u^2 (x^2-1)}]\inv\\
& =  [(1+u^2)  + (ux \sqrt{1+u^2})  ] \sum_k u^{2k} (x^2-1)^k\\
&= 1 +  \sum_{k=1}^\infty u^{2k} x^2(x^2-1)^{k-1}\\
&\quad + \left(\sum_j\binom {\frac 1 2 } j u^{2j+1} x \right)\left(\ \sum_k u^{2k} (x^2-1)^k\right).
\end{align*}

%
We conclude that 
\begin{align}
\notag 2\nu_0 &= \sigma_0 \\
\label{eq_nu}2\nu_{2k}&= \sum_{j=0}^{k-1} (-1)^j\binom {k-1} j \sigma_{2j+2}, \quad k \ge 1  \\
\notag 2\nu_{2k +1}&= \sum_{j=0}^k \binom {\frac 1 2} j \sum_{i=0}^{k-j} (-1)^{k-j-i} \binom{k-j} i \sigma_{2i+1}.
\end{align}

\section{Proof of \eqref{eq_fund_limit}}
\subsection{A kinematic inequality} The proof of \eqref{eq_fund_limit} relies on the inequality \eqref{eq_kin_ineq} below, which appears to be new.
The proofs of Lemma \ref{lem_geodesic} and of Theorem \ref{thm_abs_KF} will appear elsewhere \cite{fu_kin_ineq}. 

Given a smooth Riemannian manifold $M$, $n:= \dim M$, and a PR subset $A \subset M$,  we define the functionals $|\tau_i|(A), \ i = 0,\dots,N$, as follows. For clarity let us begin with the case where $A$ is a $C^{1,1}$ domain. For
$p \in \partial A$ and each $j$, set $\abssym_j(p)$ equal to the $j$th
elementary symmetric function in the {\it absolute values} of the principal curvatures of $\partial A$ at $p$. Then
\begin{equation}
|\tau_i|(A) := \int_{\partial A} \abssym_{n-i-1}(p) \, d\vol_{n-1}p
\end{equation}
Clearly $|\tau_i|(A) = \tau_i(A)$ if $A $ is geodesically convex.

We may extend this construction to general $A \in PR$ as follows. Recall that the unit tangent bundle $\U M$ is a contact manifold, and that the contact 
distribution $V_{\bar p}$ at each $\bar p\in \U M$ is canonically isomorphic  to $T_pM\oplus T_pM$ where $p = \pi_{M}(\bar p)$ as a symplectic
vector space. 
The normal cycle $\nc(A)$ of a PR set $A\subset M$ is given by
integration over a Legendrian Lipschitz submanifold $\operatorname{nor}(A)\subset \U M$. This implies that each tangent space $T_{\bar p} \operatorname{nor}(A)$ is a Lagrangian subspace of $V_{\bar p}$ (if $A$ is a $C^2$ body then this space is the graph of the shape operator of 
$\partial A$ at $p$). Let $m$ be the measure given by restricting Hausdorff $(n-1)$-dimensional measure
to $\operatorname{nor}(A)$.

Put  $\Gr_{n-1}^+(W)$ for the  Grassmannian of oriented  $(n-1)$-dimensional subspaces of a Euclidean space $W$, and 
$$
\Pluck: \Gr_{n-1}^+(W)\to \sS(\wedge^{n-1}W)
$$
for the spherical  Pl\"ucker embedding (i.e. into the unit sphere of $\wedge^{n-1}W$ rather than
the projective space), i.e.
$$
\Pluck V:= v_1\wedge\dots \wedge
v_{n-1},
$$
 where $v_1,\dots,v_{n-1}$ is a positively oriented orthonormal basis for $V$.
Then
\begin{equation}\label{eq_measure_decom}
\nc(A) = m\wedge \Pluck (T_{\bar p} \operatorname{nor}(A))
\end{equation}
in the sense that for any $(n-1)$-form $\phi$
$$
\int_{\nc(A)}\phi = \int \langle\Pluck (T_{\bar p} \operatorname{nor}(A)), \phi(\bar p)\rangle \, dm.
$$

\begin{lemma}
Let $W$ be an oriented Euclidean space of dimension $m$ and $\xi \in \bigwedge^m (W\oplus W)$ a simple
Lagrangian $m$-vector. Then for each $k=0,\dots, m$ there are pairwise orthogonal $k$-dimensional subspaces
$F_{k,i}, \ i = 0,\dots, I_k$ and constants $c_{k,i} \in \RR$ such that 
\begin{equation}\label{eq_lag_decomp}
\xi =  \sum_{k=0}^m \sum_{i= 0}^{I_k} c_{k,i} \cdot \Pluck(F_{k,i} \oplus F_{k,i}^\perp).
\end{equation}
Each quantity 
\begin{equation}
\parallel \xi\parallel_k := \sum_{i= 0}^{I_k} |c_{k,i}|
\end{equation}
is well-defined, independent of the decomposition \eqref{eq_lag_decomp}.
%
%
\end{lemma}
\begin{proof}
Let $E\subset W \oplus W$ be the Lagrangian subspace  corresponding to $\xi$, and $V \subset W$ its projection to 
the first $W$ factor. 
Then $E$ is the orthogonal direct sum of $\{0\} \oplus V^\perp$ with the graph of a self-adjoint linear map 
$\lambda: V\to V$.  The  $c_{k, i} $ are then the terms of
$k$th elementary symmetric functions of the eigenvalues of $\lambda$, multiplied by a constant depending only on
the Euclidean norm of $\xi$.
\end{proof}

Choosing orientations appropriately and comparing with \eqref{eq_def_tau} and \eqref{eq_measure_decom}, we see that 
\begin{equation}
\tau_k(A) =2^{k+1} \alpha_k \inv \alpha_{N-k-1}\inv \int \sum_i c_{k,i} \, dm
\end{equation}
and we put
\begin{equation}
|\tau_k|(A) =2^{k+1} \alpha_k \inv \alpha_{N-k-1}\inv \int \sum_i |c_{k,i} | \, dm = 2^{k+1} \alpha_k \inv \alpha_{N-k-1}\inv \int \parallel \xi_{\bar p} \parallel_k \, dm.
\end{equation}
Define the $|\sigma_k|$ similarly. Obviously $|\tau_k(A)|\le |\tau_k|(A)$.

The definitions of these functionals  obviously depend on the embedding of $A$ into $M$. For this reason let us  (for a brief moment) incorporate this into the notation by writing $|\tau^M_k|$.
\begin{lemma}\label{lem_geodesic}
If $ M\hookrightarrow N$ is a totally geodesic embedding of Riemannian manifolds then $|\tau^M_k|(A) = |\tau^N_k|(A)$ for PR sets $A \subset M$.
\end{lemma}

We then have the following kinematic inequality, mirroring \eqref{eq_tau_KF}:
\begin{theorem}\label{thm_abs_KF}  Let $M$ be a space form of dimension $N$ and $G$ its isometry group, equipped with 
a Haar measure $dg$ so that for measurable subsets $E \subset M$
$$
dg(\{g: g o \in E\})= \vol (E) .
$$
 Given PR sets $A,B \subset M$,
\begin{equation}\label{eq_kin_ineq}
\int_{G} |\tau_k(A \cap gB)|\,  dg \le \frac{\alpha_N}{2^{N+1}}\sum_{i+j =N+k} |\tau_i|(A) \  |\tau_j|(B).
\end{equation}
\end{theorem}

If $M= \Sigma^N$ and $G = O(N+1)$ equipped with the probability Haar measure then this becomes
\begin{equation}\label{eq_kin_ineq_sigma}
\int_{O(N+1)} |\sigma_k(C \cap gD)|\,  dg \le \frac 1 2 \sum_{i+j =k} |\sigma_i|(C) \  |\sigma_j|(D).
\end{equation}

{\bf Remark.} The functionals $|\tau_i|$ may be regarded as {\it total absolute curvature measures}.
A different, closely related, and seemingly more natural family of  total absolute curvature measures 
was introduced by M. Z\"ahle in \cite{zahle} (if $i=0$ then they both agree, up to scaling, with the 
total  curvature of \cite{chern-lashof}). It would be worthwhile to clarify the relation between
the two families. It also seems possible that a formally more natural  kinematic inequality might be stated 
in terms of Z\"ahle's functionals.

\subsection{Estimating $|\sigma_k|(D_N)$}  
Replacing $D$ by a tubular neighborhood of small radius and using a limiting argument, we may assume that that $D$
has $C^{1,1}$ boundary.

\begin{lemma}
\label{lem_asymp_vol}
$$\vol_{N-1}(\partial D_N)\asymp \notag  (2\pi e)^{\frac N 2}.$$
\end{lemma}
\begin{proof}
\begin{align*}
\notag \vol_{N-1}(\partial D_N)&\sim \vol_{d-1}(\partial D) \vol_{N-d} (\sS_{\sqrt N}^{{N-d}})\\
\notag &\asymp \alpha_{N-d} N^{\frac{N-d}2} \\
&\asymp  \left(\frac{2\pi e}{N-d}\right)^{\frac {N-d}2}N^{\frac{N-d}2}\\
&\asymp (2\pi e)^{\frac N 2}.
\end{align*}
\end{proof}

Let $N$ be large enough that $D \subset B^d_{\sqrt N}$. Given  $p \in \partial D$, we wish to estimate the 
 second fundamental form $II_{p;q}$ of $\partial D_N$ at a general point
 $(p;q) \in \pi_N\inv(p) \subset \Sigma^N\subset \RR^d \times \RR^{N+1-d}$.   Note first that the tangent space to $\Sigma^N$ is
\begin{equation}
T_{(p;q)} \Sigma^N = (p^\perp\times \{ 0\})\oplus \ell \oplus (\{0\} \times q^\perp) \subset
\RR^{d} \times \RR^{N-d+1}
\end{equation}
where $\ell$ denotes the 1-dimensional space spanned by  $ (-|q|^2p; |p|^2 q)$ (if $p =0$ then $\ell$ is subsumed into
the first factor).
The tangent space $T_{(p;q)} D_N $ is  clearly  the direct sum of the summand $Q:= \{0\} \times q^\perp $ with a 
codimension 1 subspace $P$ of the sum of the first two summands.
%
Then the second fundamental form splits as 
\begin{equation}
II_{p;q} = II_P \oplus r \Id_Q
\end{equation}
where $II_P$ is the second fundamental form at $p$ of $\partial D$ with respect to the metric \eqref{eq_alt_metric} below, and
$$
r = \frac{|p|}{|q|\sqrt N}\asymp \frac{|p|}N
$$
(this is is the geodesic curvature of a latitude curve at distance $|p|$ away from the equator in a 2-sphere 
of radius $\sqrt N$).
Clearly 
$$
II_P \sim II_{D,p} ,
$$
the second fundamental form of $\partial D$ at $p$, as $N\to \infty$. We conclude that $D_N \subset \Sigma^N$ has $d-1$ principal curvatures that are $\sim$ those of $D$, and $N-d$ that are $\precsim N\inv$. Therefore 
\begin{equation}
 \abssym_{k} \precsim \sum_{j=0}^{d-1} \binom{d-1} j \binom{N-d}{k-j} N^{-(k-j)} .
 \end{equation}
 For $k<d-1$ this is $\precsim 1$, while for $k\ge d-1$ we obtain
 \begin{align}
\label{eq_abssym_est} \abssym_{k} \precsim& \frac 1{(k-d+1)!}
\\
\notag&\asymp \frac {(2\pi e)^{k-d+1}}{(k-d+1)^{k-d+\frac 3 2}} \\
\notag & \precsim k^d \left(\frac {2\pi e}{k}\right)^k
\end{align}
Meanwhile
\begin{equation}
|\sigma_{k}|(D_N)= 2 N^{\frac{k-N}2} \alpha_{k-1}\inv \alpha_{N-k}\inv \int_{\partial D_N} \abssym_{k-1}
\end{equation}
where the coefficient satisfies 
\begin{align}
\notag 2 N^{\frac{k-N}2}  \alpha_{k-1}\inv \alpha_{N-k}\inv &\precsim
N^{\frac{k-N}2}  \frac{(N-k)^{\frac {N-k} 2}(k-1)^{\frac {k-1}2}}{(2\pi e)^{\frac N 2}}
\\
\label{eq_ck} &
< \frac{k^{\frac {k}2}}{(2\pi e)^{\frac N 2}}
\end{align}
With \eqref{eq_abssym_est} and Lemma \ref{lem_asymp_vol} we thus estimate 
\begin{equation}\label{eq_sigma_est}
|\sigma_k|(D_N)\precsim k^{d-\frac{k} 2}(2\pi e)^k 
\end{equation}

\subsection{Completion of the proof of \eqref{eq_fund_limit}} \label{sect_Pi_converge} Recalling \eqref{eq_KF1} we show first that the 
absolute expectations 
\begin{equation}\label {eq_L1_est}
\EE_N[|t^m(A \cap F\inv D)|]= 2^m \int_{O(N+1)} \left|u^m(\iota_N A \cap g D_N)\right|\, dg, \quad m\le n
\end{equation}
are bounded (the integrand vanishes if $m>n$ since $A \subset \sS_1^n$). By \eqref{eq_kin_ineq_sigma},
\begin{equation}\label {eq_L1_sigma_est}
  \int_{O(N+1)} \left|\sigma_k  (\iota_N A \cap g D_N)\right|\, dg 
  \le \frac 1 2 \sum_{i+j=k} |\sigma_i| (\iota_NA) |\sigma_j|(D_N).
\end{equation}
By \eqref{eq_sigma_to_u} it is enough to show that these integrals 
 are bounded for $k\ge N-n$. 
But $|\sigma_i| (\iota_NA)= |\tau_{N-i}|(A) $, which moreover vanishes if $i < N-n$. Thus $j\le n$ in every nonzero term 
of the sum, so the bound \eqref{eq_sigma_est} yields \eqref{eq_L1_sigma_est} and thereby also 
\eqref{eq_L1_est}.

Put for conciseness $z(F):=t^m(A \cap F\inv D), \ F \in \calL$. Recalling the notation of Section \ref{sec_approx} we have just shown
the existence of a constant $C$ such that
\begin{equation}\label{eq_unif_bound}
\int_{\calL} |z(F) |\, \Pi_N(F) \, dF \le C <\infty
\end{equation}
for all $N$.
By  \cite{diac_et_al}
\footnote{The first point is Prop. 2.2, op. cit. The second point follows from  (33) of the Appendix, op. cit.},
\begin{itemize}
\item $\Pi_N\to \Pi_\infty $ in $L_1(\calL)$, hence (possibly taking a subsequence) pointwise a.e., and
\item there is a constant $C$ such that $\Pi_N \le C \Pi_\infty$ for all $N$.
\end{itemize}
With \eqref{eq_unif_bound}, the first point implies that $\int_{\calL} |z(F) |\, \Pi_\infty(F) \, dF  <\infty$ by 
Fatou's Lemma. The second point
thus implies that we may apply the dominated convergence theorem, with dominating function $C |z|\Pi_\infty$, to 
conclude that 
\begin{equation}
\lim_{N\to \infty } \int_{\calL} z(F) \, \Pi_N(F) \, dF  = \int_{\calL} z(F) \, \Pi_\infty(F) \, dF 
\end{equation}
as claimed.

 \section{Proof of \eqref{eq_limit_of_nu}}

\subsection {Tubes in $\RR^d$ versus tubes in $\Sigma^N$}

If we endow the open ball $B^d_{\sqrt N}$ with the metric
\begin{equation}\label{eq_alt_metric}
ds_N^2 := dx^2+ \frac{\langle x, dx\rangle^2}{N-|x|^2}
\end{equation}
then the projection map $\pi_N: \Sigma^N \to \RR^d$
becomes a Riemannian submersion off of the equator. Restricted to any compact subset of $\RR^d$, the metrics $ds_N^2 \to dx^2$ in the $C^\infty$ topology. It follows that $D$ has uniformly positive reach with respect to $ds_N^2$ among those $N$ for which $D \subset B^d_{\sqrt N}$, and that the distance functions
$\delta_N(y):= \dist_N(y,D)$ converge locally uniformly to $\delta(y):= \dist(y,D)$ and
\begin{align*}
\Sigma^N \supset \tube (D_N,r) & = \pi_N\inv (\tube_N(D,r))
\end{align*}
where $\tube_N$ denotes the tubular neighborhood with respect to \eqref{eq_alt_metric}.

Put $\rho_N$ for the density of the projection under $\pi_N$ of the uniform probability  measure on $\Sigma^N$.
What we have referred to as Poincar\'e's limit  (easily derived using the coarea formula) states that
\begin{equation}\label{eq_poincare}
\rho_N(x)\to (2\pi)^{-\frac d 2} e^{-\frac{|x|^2}2}
\end{equation}
locally uniformly for $x \in \RR^d$ as $N\to \infty$.
It follows that for 
each $r >0$
\begin{equation}
 \int_{\tube_N(D,r)} \rho_N(x)\, dx \to \int_{\tube(D,r)} (2\pi)^{-\frac d 2} e^{-\frac{|x|^2}2}\, dx = \gamma_0(\tube(D, r)).
\end{equation}

To complete the proof of \eqref{eq_limit_of_nu} we recall \eqref{eq_pN_chi} and
note that since the $D_N$ have uniformly positive reach $\ge r_0>0$
\begin{equation}\label{eq_tube}
\vol_{\Sigma^N}(\tube_{\Sigma^N}(D_N,r))= p_N(\chi)(\Sigma^N_r, D_N) =  \sum_{k=0}^N u^k(\Sigma^N_r) \cdot \nu_k(D_N)
\end{equation}
for all $0\le r\le r_0$, $\Sigma^N_r$ denotes the geodesic ball of radius $r$ in $\Sigma^N$.


\subsection {Intrinsic volumes of geodesic balls in $\Sigma^N$} 
We wish to show that this last sum
is approximately a truncated power series, with radius of convergence bounded below independent 
of $N$. 

\begin{proposition}\label{prop_growth} Let $\Sigma_r^N$ denote the geodesic ball of radius $r$ in $\Sigma^N$. Then for fixed $k$
\begin{equation}
u^k(\Sigma_r^N) 
\sim\frac{\omega_k}{ \left(2\pi  \right)^{\frac k 2}}  r^k 
\end{equation}
 locally uniformly in $r$ as $N\to \infty$, and
\begin{equation}
u^k(\Sigma_r^N) \precsim \left(\frac {e}k\right)^{\frac {k+1} 2}  r^k
\end{equation}
for $0<r<1$.

\end{proposition}
\begin{proof}
Observe first that by \eqref{eq_t_mu} and \eqref{eq_intrinsic_vol}
\begin{equation}\label{eq_exact_tk}
u^k(B^N_1) \begin{cases}\sim \frac{\omega_k}{ \left(2\pi  \right)^{\frac k 2}}  \quad  {\rm for \ fixed} \ k\\
\precsim \left(\frac e k\right)^{\frac {k+1} 2}.
\end{cases}
\end{equation}

Turning to the matter at hand, we treat separately the cases where $k,N$ have the same or opposite parity, making use of the fact that 
\begin{equation}\label{eq_bdry}
2 t^k(M) =  t^k(\partial M)
\end{equation}
if $M$ is a compact manifold with boundary and $k+\dim M $ is odd.
In particular, $t^k(M)=0$ if $\partial M = 0$ in this case.

Put $L^N_r:= B^{N+1}_{\sqrt N}\cap \{(x_0,\dots,x_N): x_0 \ge \sqrt N \cos \frac r{\sqrt N}\}$ . Thus 
$L^N_r$ is the intersection of the ball with a half space, specifically a convex body with boundary equal to the union of $\Sigma^N_r $ with an $N$-dimensional disk $C^N_r$ of radius ${\sqrt N\sin \frac {r}{\sqrt N}}$, intersecting in the sphere $\partial C^N_r$ of dimension $N-1$.

Suppose first that $k$ and $N$ have opposite parity. Then \eqref {eq_bdry} gives
\begin{align*}
0 = t^k(\partial L^N_r) &= t^k(\Sigma^N_r ) + t^k( C^N_r )- t^k (\partial C^N_r) \\
&= t^k(\Sigma^N_r ) -t^k( C^N_r )
\end{align*}
from which the desired estimates follow using \eqref{eq_exact_tk}.

If $k$ and $N$ have the same parity, note first that
\begin{equation*}
C^N_r\subset L^N_r\subset   I^N_r  \times C^N_r
\end{equation*}
where $I^N_r$ is the interval $[\sqrt N \cos \left(\frac r{\sqrt N}\right), \sqrt N]$, so by
monotonicity of $\mu_k$ under inclusions of convex sets
\begin{equation}\label{eq_inclusions}
\mu_k(C^N_r)\le \mu_k(L^N_r)\le\mu_k( I^N_r  \times C^N_r).
\end{equation}
Since by \eqref{eq_bdry}
$$\mu_k(L^N_r) = \frac 1 2 \mu_k(\partial L^N_r)= \frac 1 2 (\mu_k(\Sigma^N_r)+ \mu_k(C^N_r)- \mu_k(\partial C^N_r))=  \frac 1 2 (\mu_k(\Sigma^N_r)+ \mu_k(C^N_r))
$$
and by \eqref{eq_mu1} and \eqref{eq_product}
\begin{align*}
  \mu_k(C^N_r\times I^N_r) 
&= \mu_k(C^N_r) \mu_0(I^N_r) + \mu_{k-1}(C^N_r) \mu_1(I^N_r) \\
&=
\mu_k(C_r^N) + \mu_{k-1}(C_r^N) \sqrt N\left(1- \cos \left(\frac r{\sqrt N}\right)\right)\\
&\le \mu_k( C^N_r) + \frac {r^2} {\sqrt N} \mu_{k-1}( C^N_r)
\end{align*}
we find that 
\eqref{eq_inclusions} yields
$$
\mu_k( C^N_r) \le 
\mu_k(\Sigma^N_r) \le \mu_k(C^N_r) + \frac {2r^2} {\sqrt N} \mu_{k-1}( C^N_r)
$$
Multiplying by $\frac {\omega_k k!}{(4N)^{\frac k 2} \pi^k}$,
$$u^k(C^N_r) \le 
u^k(\Sigma^N_r) \le 
u^k( C^N_r) +\frac{r^2 k \omega_k}{\pi \omega_{k-1} N}u^{k-1}( C^N_r)
$$
where by \eqref{eq_exact_tk} and by \eqref{eq_omega} the last term
$
\precsim \frac {r^{k+1}} N \left(\frac e{k}\right)^{\frac{k}2 +1}$.
\end{proof}

\subsection{Conclusion of the proof of \eqref{eq_limit_of_nu}} 
We note that 
\begin{lemma}
The function $v(r):= \gamma_0(\tube(D,r))$ is analytic at $r=0$.
\end{lemma}
\begin{proof}
Put
$$
\exp: \RR^d \times \sS^{d-1}_1 \times[0,1]\to \RR^d, \quad \exp(x,v,s):= x+ sv.
$$
If $ r <$ the reach of $D$ then
$$
\gamma_0(\tube(D,r)) = \gamma_0(D) + \frac 1{\sqrt{2\pi d}} \int_{\nc(D)\times[0,r]} \exp^*(e^{-\frac {|x|^2 }2}
dx_1\wedge\dots \wedge dx_d).
$$

\end{proof}

Combining \eqref{eq_nu} and \eqref{eq_sigma_est}
\begin{align*}
|\nu_{2\ell}(D_N)|&\le \sum_{j=0}^{\ell-1} \binom {\ell-1} j |\sigma_{2j+2}|(D_N)\\
&\precsim \sum_{j=0}^{\ell-1} \binom {\ell-1} j (2j+2)^{d-j-
\frac 3 2} e^{2j+2}\\
&\le (2\ell)^de^2(1+e^2)^{\ell-1}
\end{align*}
for $\ell \ge 1$, and, since
$
\sum  \left| \binom {\frac 1 2} j \right| x^j = 2- (1-x)^\frac 1 2
$ for $|x|<1$, 
\begin{align*}
|\nu_{2\ell +1}(D_N)|&\le \sum_{j=0}^\ell \left| \binom {\frac 1 2} j \right|\sum_{k=0}^{\ell-j}  \binom{\ell-j} k| \sigma_{2k+1}|(D_N) \\
&\precsim  \sum_{j=0}^\ell  \left| \binom {\frac 1 2} j \right| \sum_{k=0}^{\ell-j}  \binom{\ell-j} k (2k+1)^{d-k-1} e^{2k+1} \\
&\le (2\ell+1)^d e \sum_{j=0}^\ell  \left| \binom {\frac 1 2} j \right| \sum_{k=0}^{\ell-j}  \binom{\ell-j} k  e^{2k} \\
&=  (2\ell+1)^d e \sum_{j=0}^\ell  \left| \binom {\frac 1 2} j \right| (1+e^2)^{\ell-j}\\
&\le (2\ell+1)^d e (1+e^2)^{\ell}  \left(2-\sqrt{1- (1+e^2)\inv}\right)
\end{align*}

With Proposition \ref{prop_growth} we therefore see that 
there are positive constants $b_0,b_1,\dots$ such that 
\begin{itemize}
\item $\sum b_i r^i$ has 
positive radius of convergence, and
\item $|u^k(\Sigma^N_r) \cdot \nu_k(D_N)| \le b_k r^k$ for all $k,N$ and $0<r<1$.
\end{itemize}
Recalling \eqref{eq_tube} it follows that for each $m$ the Maclaurin polynomial of degree $m$ of $\gamma_0(\tube(D,r))$ is  
$$ \lim_{N\to \infty} \sum_{k=0}^m u^k(\Sigma^N_r) \cdot \nu_k(D_N)
= \lim_{N\to \infty} \sum_{k=0}^m \frac{\omega_k }{\left(2\pi \right)^{\frac k 2} }  \cdot \nu_k(D_N) r^k.
$$
The desired relation  \eqref{eq_limit_of_nu} now follows by induction on $m$.

\section{Concluding remarks}

\subsection{Beyond PR}  Even though  the natural domain of classical integral geometry remains mysterious, we know  \cite {fu_IG} that these ideas apply to a number of  settings (IGregularity classes) beyond PR, such as
compact subanalytic sets and WDC sets. We expect that the results above apply 
in those cases as well, possibly with only minimal modifications of the proofs.

\subsection{The Gaussian convolution algebra} S. Alesker and A. Bernig \cite{ al_ber} have introduced
a natural algebraic operation $(\phi,\psi)\mapsto  \phi*\psi$ called {\it convolution} on suitable spaces of valuations on
euclidean space $\RR^d$. It is natural to ask how the
Gaussian intrinsic volumes $\gamma_i$ fit into this scheme. 

The first observation is that  \cite{al_ber} only develops the theory for compactly supported valuations, which do not include the Gaussian intrinsic volumes. However, one suspects that they belong to a suitably defined space of ``Schwartz class" valuations to which the Alesker-Bernig
definition may be extended. 

Moreover, Theorem 2 of \cite{al_ber} suggests the following conjecture.
Define smooth valuations $\gamma_i^\lambda$ on $\RR^d$ analogously to the $\gamma_i$,
with the unit variance of the centered Gaussian distribution replaced by $\lambda$. Then
\begin{equation}
\gamma^\lambda_i*\gamma_j^\mu = \gamma_{i+j}^{\sqrt{\lambda^2 +\mu^2}}.
\end{equation}

\subsection{The full Adler-Taylor theorem} Adler and Taylor framed the theorem differently,
in terms of an isometric embedding $\Phi=(\phi_0,\dots, \phi_n)$ of a Riemannian manifold
$M$ into $\sS^n_1$, and the resulting random function $\sum_{i=0}^n \xi_i \phi_i$ on $M$ (our framing
is equivalent). The full extent of their theorem goes well beyond this however, giving 
similar conclusions for random functions $\sum_{i=0}^\infty \xi_i\phi_i$ corresponding to an
isometric embedding $(\phi_0,\phi_1,\dots)$ of $M$ into the unit sphere of $\ell^2$. Such random
functions may be characterized axiomatically as the {\it Gaussian random fields} described 
in \cite{adl-tay} and \cite{adl_tay_book}. As Adler and Taylor emphasized, the resulting Lipschitz-Killing valuations $t^k$ on
$M$ are Riemannian invariants, a fact that we call the Weyl principle.

One of the main points of the present paper is to emphasize the essentially valuation-theoretic
nature of Adler-Taylor theory. It would be desirable to give a proof also of the broader Adler-Taylor
theorem in these terms. 

Another intriguing prospect of this subject is a possible avenue to prove the
Weyl principle for PR sets $A$-- that is, that the restrictions to $A$ of the Lipschitz-Killing valuations to any
such $A$ depend only on its intrinsic metric structure. This is known if $A$ is suitably stratified, but this is not the case for general PR sets.

\end{document}